\documentclass[onefignum,onetabnum]{siamart171218}

\ifpdf
\hypersetup{
  pdftitle={Large deviation theorem for branches of the random binary tree in the Horton-Strahler analysis},
  pdfauthor={K. Yamamoto}
}
\fi

\usepackage{amsmath, amssymb}

\newtheorem{rem}{Remark}
\newtheorem{prop}{Proposition}

\title{Large deviation theorem for branches of the random binary tree in the Horton-Strahler analysis}

\author{Ken Yamamoto\thanks{Department of Physics and Earth Sciences, Faculty of Science, University of the Ryukyus, Senbaru, Okinawa, Japan (\email{yamamot@sci.u-ryukyu.ac.jp}).}}

\newcommand{\avg}[1]{E\left[#1\right]} % average
\newcommand{\cvgdist}{\xrightarrow{D}} % convergence in distribution

\newcommand{\floor}[1]{\lfloor#1\rfloor} % floor function

\begin{document}

\maketitle

% REQUIRED
\begin{abstract}
The Horton-Strahler analysis is a graph-theoretic method to measure the bifurcation complexity of branching patterns, by defining a number called the order to each branch.
The main result of this paper is a large deviation theorem for the number of branches of each order in a random binary tree.
The rate function associated with a large deviation cannot be derived in a closed form;
instead, asymptotic forms of the rate function are given.
\end{abstract}

% REQUIRED

% REQUIRED

\section{Introduction}\label{sec1}
The topological analysis of branching patterns or objects began with hydrological research on river networks.
Horton proposed a systematic method to assign a number (called the \emph{order}) to each stream based on the join of streams~\cite{Horton}.
Horton's law of stream numbers is an empirical relation stating that the number of streams of order $r$ decreases geometrically with $r$.
Horton's method partially needs information about spatial configuration of the river network such as stream lengths and junction angles.
Strahler refined Horton's method so that the order is defined by a purely graph-theoretic way~\cite{Strahler}.

Strahler's ordering method for a binary tree is composed of the following three rules.
\begin{enumerate}
\item The leaf nodes (degree-one nodes) are defined to have order 1.
\item A node whose children have different order $r_1$ and $r_2$ ($r_1\ne r_2$) has order $\max\{r_1, r_2\}$.
\item A node whose two children have the same order $r$ has order $r+1$.
\end{enumerate}
We define a \emph{branch} of order $r$ as a maximal connected path whose constituent node(s) all have the same order $r$.
For a binary tree $\tau$ having $n$ leaves, let $S_{r,n}(\tau)$ denote the number of its order-$r$ branches.
Further, based on Strahler's ordering, Tokunaga~\cite{Tokunaga} established a method, called the Tokunaga indexing, to describe the structure of side-branching.
The Horton-Strahler analysis, based on the branch order, has been applied to a wide variety of branching patterns and structures, such as botanical trees~\cite{Leopold} and blood vessels~\cite{VanBavel} in biology, register allocation in computer science~\cite{Devroye}, cracks in material engineering~\cite{Djordjevic}, and complex network analysis~\cite{Guimera}.

%In modeling a river network as a binary tree, nodes the sources, the estuary, and the junction points
%each node has zero or two children
In this paper, we focus on \textit{rooted}, \textit{planar}, \textit{full} binary trees~\cite{Stanley}.
This class of binary trees appears naturally in modeling a river network.
A special node corresponding to the estuary is called the \textit{root}; each stream has a flow direction towards it.
A river network is embedded in the ground surface, so the corresponding tree is \textit{planar}; formally, a planar binary tree is defined as a rooted binary tree with right and left directions assigned to each pair of children of the same parent.
If we set only the junction points as internal nodes, each node has either zero or two children; this type of binary tree is called a \textit{full} binary tree.
We let $\Omega_n$ denote the set of planar full binary trees having $n$ leaves.
The number of distinct trees in $\Omega_n$ is expressed as
\[
\vert\Omega_n\vert=\frac{(2n-2)!}{n!(n-1)!},
\]
and this combinatorial quantity is known as the $(n-1)$th Catalan number~\cite{Stanley}.
For example, $\Omega_3$ consists of two binary trees, each of which has three branches of order 1 and one branch of order 2 (see Fig.~\ref{fig1} for reference).
A probability space formed by introducing the uniform probability measure on $\Omega_n$ is referred to as the random binary tree model (or \emph{random model} for short), introduced by Shreve~\cite{Shreve}.
Note that $S_{r,n}$ is a random variable on the random model.

\begin{figure}[t!]\centering
\includegraphics[width=\textwidth]{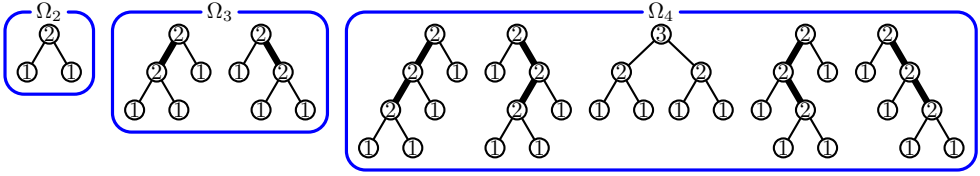}
\caption{
Illustration of $\Omega_2$, $\Omega_3$, and $\Omega_4$, respectively containing one, two, and five binary trees.
Each number on the nodes is Strahler's order.
The bold edges represent branches of order 2.
}
\label{fig1}
\end{figure}

A striking feature common to some stochastic tree models is self-similarity, which means the invariance under the operation of pruning (cutting leaves)~\cite{Peckham}.
A stochastic tree model is said to be Horton self-similar if it satisfies Horton's law, and self-similarity involving side-branching structure is called Tokunaga self-similarity.
The random model, the Tokunaga model~\cite{Tokunaga}, and the random self-similar network~\cite{Veitzer} are well-known self-similar tree models.
%Stochastic branching processes~\cite{Harris}, such as the Galton-Watson process, and coalescent processes~\cite{Bertoin} naturally yield tree structures.
%Self-similarity of the critical binary Galton-Watson process~\cite{Burd}, level set trees of Markov chain~\cite{Zaliapin}, and Kingman's coalescent trees~\cite{Kovchegov} was proved so far.
%The structural complexity of self-similar trees was characterized by using the entropy rate~\cite{Chunikhina}.
For the development and related topics of the self-similarity of random trees, see Kovchegov and Zaliapin~\cite{Kovchegov2019}.

The main subject of this paper is the asymptotic property of $S_{r,n}$.
For any function $f:\mathbb{N}\cup\{0\}\to\mathbb{R}$, $f(S_{r,n}(\cdot))$ is a real-valued random variable on the random model.
By applying the pruning operation, a recursive relation
\begin{equation}
\avg{f(S_{r+1,n})}=\frac{n!(n-1)!(n-2)!}{(2n-2)!}\sum_{m=1}^{\floor{n/2}}\frac{2^{n-2m}}{(n-2m)!m!(m-1)!}\avg{f(S_{r,m})},
\label{eq:connect}
\end{equation}
between the averages of two adjoining orders $r$ and $r+1$ holds~\cite{Yamamoto2010},
where $\avg{\cdot}$ denotes the average on the random model.
In this paper, we mainly study the case where $f$ is an exponential function.

Wang and Waymire~\cite{Wang} proved the central limit theorem for $S_{2,n}$:
\begin{equation}
\sqrt{n}\left(\frac{S_{2,n}}{n}-\frac{1}{4}\right)\cvgdist N\left(0, \frac{1}{16}\right),
\quad n\to\infty,
\label{eq:cltWang}
\end{equation}
where `$\cvgdist$' denotes convergence in distribution, and $N(\mu, \sigma^2)$ is the normal distribution with mean $\mu$ and variance $\sigma^2$.
Recently, Yamamoto~\cite{Yamamoto2017} obtained two generalized forms of Eq.~\eqref{eq:cltWang} as
\begin{equation}
\sqrt{n}\left(\frac{S_{r+1,n}}{n}-\frac{1}{4^r}\right)\cvgdist N\left(0, \frac{4^r-1}{3\cdot16^r}\right),
\quad n\to\infty
\label{eq:cltYamamoto1}
\end{equation}
and
\begin{equation}
\sqrt{n}\left(\frac{S_{r+1,n}}{S_{r,n}}-\frac{1}{4}\right)\cvgdist N(0, 4^{r-3}),
\quad n\to\infty
\label{eq:cltYamamoto2}
\end{equation}
for each $r=1,2,\ldots$
These results are both reduced to Eq.~\eqref{eq:cltWang} when $r=1$.
Equation~\eqref{eq:cltYamamoto1} implies that $S_{r+1,n}/n$ converges in probability to $4^{-r}$ as $n\to\infty$, which is compared with Horton's law.

It is worth pointing out that the central limit theorem~\eqref{eq:cltYamamoto2} can be derived easily by the pruning operation~\cite{Burd, Kovchegov2019}.
Since pruning a binary tree $\tau\in\Omega_n$ $r-1$ times yields a binary tree having $S_{r,n}(\tau)$ leaves,
Eq.~\eqref{eq:cltWang} immediately implies
\[
\sqrt{S_{r,n}}\left(\frac{S_{r+1,n}}{S_{r,n}}-\frac{1}{4}\right)\cvgdist N\left(0, \frac{1}{16}\right),
\quad n\to\infty.
\]
Considering $S_{r,n}/n\to 4^{1-r}$ along with this equation, we obtain the central limit theorem~\eqref{eq:cltYamamoto2}.

As for $S_{2,n}$, the following large deviation theorem was demonstrated~\cite{Wang}.
\begin{theorem}[Large deviation theorem for $S_{2,n}$]
For the random model,
\[
\lim_{n\to\infty}\frac{1}{n}\log P\left(\frac{S_{2,n}}{n}>y\right)=-I(y),\quad y\in\left(\frac{1}{4}, \frac{1}{2}\right)
\]
and
\[
\lim_{n\to\infty}\frac{1}{n}\log P\left(\frac{S_{2,n}}{n}<y\right)=-I(y),\quad y\in\left(0, \frac{1}{4}\right),
\]
where the rate function $I(y)$ is given by
\begin{equation}
I(y)=(4y-1)\tanh^{-1}(4y-1)-\log(\cosh(\tanh^{-1}(4y-1))).
\label{eq:rate}
\end{equation}
\label{thm:LDTWang}
\end{theorem}
For the proof of this theorem, the following general result on large deviation properties is important.
\begin{theorem}[Cox and Griffeath~\cite{Cox}]
Let $(X_1, X_2.\ldots)$ be a sequence of random variables and let
\[
\varphi_n(\xi) = a_n^{-1}\log\avg{\exp(\xi X_n)},
\]
where $\{a_n\}$ is a sequence of positive numbers such that $a_n\to\infty$.
Assume that on the interval $(\xi_{-},\xi_{+})\ni0$,
\[
\lim_{n\to\infty}\varphi_n(\xi)=\varphi_{\infty}(\xi)<\infty,
\]
where $\varphi_{\infty}(\xi)$ is strictly convex and $C^2$ on $(\xi_{-},\xi_{+})$.
If $\varphi'_n$ is convex on $[0, \xi_{+})$ and $\lim_{n\to\infty}\varphi''_n(0)=\sigma^2=\varphi''_\infty(0)$, then
\[
\lim_{n\to\infty}a_n^{-1}\log P\left(\frac{X_n}{a_n}>y\right)=-I(y),\quad y\in(\mu, \alpha_+)
\]
and
\[
\lim_{n\to\infty}a_n^{-1}\log P\left(\frac{X_n}{a_n}<y\right)=-I(y),\quad y\in(\alpha_-, \mu),
\]
where $\mu=\varphi'_\infty(0)$, $\alpha_-=\varphi'_\infty(\xi_-+)$, $\alpha_+=\varphi'_\infty(\xi_+-)$, and $I(y)$ is the Legendre transform of $\varphi_\infty(\xi)$.
In addition, the central limit theorem
\[
\frac{X_n-\avg{X_n}}{\sqrt{a_n}}\cvgdist N(0, \sigma^2),\quad n\to\infty
\]
holds.
\label{thm:LDT}
\end{theorem}
For systematic treatment of large deviation theory, see Ellis~\cite{Ellis} and Deuschel and Stroock~\cite{Deuschel} for example.
Wang and Waymire~\cite{Wang} derived
\begin{equation}
\lim_{n\to\infty}\frac{1}{n}\log\avg{\exp(\xi S_{2,n})}
=\frac{\xi}{4}+\log\left(\cosh\frac{\xi}{4}\right)
=:\varphi(\xi),
\label{eq:phi}
\end{equation}
for any $\xi\in\mathbb{R}$, which leads to the proof of Theorem~\ref{thm:LDTWang} (with $a_n=n$ in Theorem~\ref{thm:LDT}).
The rate function $I(y)$ given in Eq.~\eqref{eq:rate} is the Legendre transform of $\varphi$.
Furthermore, owing to $\varphi''(0)=1/16$, Theorem~\ref{thm:LDTWang} directly implies the central limit theorem~\eqref{eq:cltWang} via Theorem~\ref{thm:LDT}.

Unlike Eq.~\eqref{eq:cltWang}, central limit theorems~\eqref{eq:cltYamamoto1} and \eqref{eq:cltYamamoto2} were obtained by the asymptotic properties of the characteristic functions of $S_{r+1,n}/n$ and $S_{r+1,n}/S_{r,n}$~\cite{Yamamoto2017}.
Thus, a natural problem is to establish the large deviation theorems corresponding to Eqs.~\eqref{eq:cltYamamoto1} and \eqref{eq:cltYamamoto2}.
In this paper, a large deviation theorem connected to Eq.~\eqref{eq:cltYamamoto1} is formulated and proved.

\section{Main result}
\begin{lemma} % theorem, lemma, proposition?
For $r=1,2,\ldots$ and $\xi\in\mathbb{R}$,
\[
\lim_{n\to\infty}\frac{1}{n}\log\avg{\exp(\xi S_{r+1,n})}=\overbrace{\varphi\circ\cdots\circ\varphi}^{r}(\xi)=\varphi^r(\xi),
\]
where the function $\varphi$ is introduced in Eq.~\eqref{eq:phi}.
\label{mainresult}
\end{lemma}
This is a generalized result of Eq.~\eqref{eq:phi}.
We give the proof of this lemma in the next section.

By Lemma~\ref{mainresult}, we can prove the following large deviation theorem for $S_{r+1,n}$.
\begin{theorem}[Large deviation theorem for $S_{r+1,n}$]
For $r=1,2,\ldots$,
\[
\lim_{n\to\infty}\frac{1}{n}\log P\left(\frac{S_{r+1, n}}{n}>y\right)=-I_r(y),\quad y\in\left(\frac{1}{4^r},\frac{1}{2^r}\right)
\]
and
\[
\lim_{n\to\infty}\frac{1}{n}\log P\left(\frac{S_{r+1, n}}{n}<y\right)=-I_r(y),\quad y\in\left(0, \frac{1}{4^r}\right),
\]
where the rate function $I_r(y)$ is the Legendre transform of $\varphi^r(\xi)$.
\label{thm:main}
\end{theorem}
Note that this theorem includes Theorem~\ref{thm:LDTWang} as a special case of $r=1$.

\begin{proof}
We can complete the proof by substituting $\varphi^r(\xi)$ in Lemma~\ref{mainresult} for $\varphi_\infty(\xi)$ in Theorem~\ref{thm:LDT}.
Since the function $\varphi(\xi)$ is strictly increasing, strictly convex and $C^2$ on $\mathbb{R}$, its composite $\varphi^r(\xi)$ also possesses these properties.
Hence, $\xi_\pm=\pm\infty$ for any $r$.
By the chain rule and $\varphi'(\xi)=[1+\tanh(\xi/4)]/4$, the derivative of $\varphi^r(\xi)$ is
\begin{equation}
(\varphi^r)'(\xi)
%=\varphi'(\varphi^{r-1}(\xi))\cdot\varphi'(\varphi^{r-2}(\xi))\cdot\cdots\cdot\varphi'(\xi)
=\prod_{j=0}^{r-1}\varphi'(\varphi^j(\xi))
=\prod_{j=0}^{r-1}\frac{1+\tanh(\varphi^j(\xi)/4)}{4}.
\label{eq:dphir}
\end{equation}
Owing to $\varphi(0)=0$ and $\varphi(\pm\infty)=\pm\infty$, we obtain
\[
\mu=(\varphi^r)'(0)=\frac{1}{4^r},\quad
\alpha_-=(\varphi^r)'(-\infty)=0,\quad
\alpha_+=(\varphi^r)'(\infty)=\frac{1}{2^r}.
\]
Therefore, the proof is complete.
\end{proof}

\begin{rem}
As a consequence of Theorem~\ref{thm:main}, the central limit theorem~\eqref{eq:cltYamamoto1} holds straightforwardly from Lemma~\ref{mainresult} and Theorem~\ref{thm:LDT}.
By differentiating Eq.~\eqref{eq:dphir} again and applying the Leibniz rule, the second derivative of $\varphi^r$ is
\begin{align}
(\varphi^r)''(\xi)
&=\varphi'(\varphi^{r-1}(\xi))\cdots\varphi'(\xi)\sum_{k=0}^{r-1}\frac{(\varphi'\circ\varphi^k)'(\xi)}{\varphi'(\varphi^k(\xi))}\notag\\
&=(\varphi^r)'(\xi)\sum_{k=0}^{r-1}\frac{\varphi''(\varphi^k(\xi))}{\varphi'(\varphi^k(\xi))}\prod_{l=0}^{k-1}\varphi'(\varphi^l(\xi)).
\label{eq:ddphir}
\end{align}
Using $\varphi(0)=0$, $\varphi'(0)=1/4$, $\varphi''(0)=1/16$, and ($\varphi^r)'(0)=4^{-r}$, we obtain
\[
(\varphi^r)''(0)=\frac{4^r-1}{3\cdot16^r}.
\]
Thus, the central limit theorem~\eqref{eq:cltYamamoto1} is derived.
\end{rem}

Lemma~\ref{mainresult} and Theorem~\ref{thm:main} indicate that the order $r$ appears in the composition $\varphi^r$.
The author believes that this regularity implies self-similarity of trees from the perspective of large deviation theory.

A large deviation formalism of Eq.~\eqref{eq:cltYamamoto2} is not studied in this paper, and is an open problem.

\section{Proof of Lemma~\ref{mainresult}}
This section is mainly devoted to the proof of Lemma~\ref{mainresult}.

First, we show Lemma~\ref{mainresult} for $r=1$ (corresponding to $S_{2,n}$).
This case~\eqref{eq:phi} was already proved by Wang and Waymire~\cite{Wang}, but we employ a formula different from theirs.
Our method has the major advantage that we can easily extend to $r\ge2$.

We need to estimate $\avg{\exp(\xi S_{2,n})}$, which is the moment generating function of $S_{2,n}$.
By setting $r=1$ and $f(S_{r,n})=\exp(\xi S_{r,n})$ in Eq.~\eqref{eq:connect},
\[
\avg{\exp(\xi S_{2,n})}=\frac{n!(n-1)!(n-2)!}{(2n-2)!}\sum_{m=1}^{\floor{n/2}}\frac{2^{n-2m}}{(n-2m)!m!(m-1)!} e^{\xi m}.
\]
This sum can be calculated exactly using the Gauss hypergeometric function~\cite{Yamamoto2008},
but here we perform an asymptotic analysis using a saddle-point method.

Letting $m=\beta n$ ($0<\beta<1/2$) to replace the sum by integral about $\beta$, and using Stirling's approximation
\[
N!\sim\sqrt{2\pi N}\left(\frac{N}{e}\right)^N,
\]
we get
\begin{align}
\avg{\exp(\xi S_{2,n})}&\sim\frac{2}{\sqrt{\pi n}}\int_0^{1/2}\frac{1}{\sqrt{1-2\beta}}\left(\frac{e^{\beta\xi}}{2(1-2\beta)^{1-2\beta}\beta^{2\beta}4^\beta}\right)^n d \beta \notag\\
&=\frac{2}{\sqrt{\pi n}}\int_0^{1/2}\frac{1}{\sqrt{1-2\beta}}\exp\left(n g(\beta;\xi)\right)d \beta,
\label{eq:saddlepoint}
\end{align}
where
\[
g(\beta; \xi):=\xi \beta-(1-2\beta)\log(1-2\beta)-2\beta\log \beta-\beta\log 4-\log2,
\]
and `$\sim$' denotes the asymptotic equality in the sense that the ratio between both hand sides tends to unity as $n\to\infty$.
One can easily confirm that the function $g(\beta; \xi)$ takes a maximum value at
\[
\beta_0=\frac{e^{\xi/4}}{4\cosh(\xi/4)},
\]
thereby
\[
\avg{\exp(\xi S_{2,n})}\sim \frac{2}{\sqrt{\pi}n}\frac{1}{\sqrt{1-2\beta_0}}\exp(n g(\beta_0; \xi))\sqrt{\frac{2\pi}{-g''(\beta_0; \xi)}}.
\]
Therefore,
\[
\lim_{n\to\infty}\frac{1}{n}\log\avg{\exp(\xi S_{2,n})}=g(\beta_0; \xi)=\frac{\xi}{4}+\log\left(\cosh\frac{\xi}{4}\right).
\]

\begin{rem}
Equation~\eqref{eq:phi} was previously obtained~\cite{Wang} by using a saddle-point method to
\begin{align*}
\avg{\exp(\xi S_{2,n})}&=\sum_{k=0}^{\floor{n/2}}\frac{(n-k)!\vert\Omega_{n-k}\vert(e^\xi-1)^k}{(n-2k)!\vert\Omega_n\vert k!}\\
&=\frac{n!(n-1)!}{(2n-2)!}\sum_{k=0}^{\floor{n/2}}\frac{(2n-2k-2)!(e^\xi-1)^k}{(n-2k)!(n-k-1)!k!}.
\end{align*}
However, as noted in \cite{Wang}, this procedure needs to treat the two cases where $e^\xi-1$ is positive and where $e^\xi-1$ is negative separately.
Moreover, it seems to be difficult to extend their method to general $S_{r+1,n}$.
In this light, our method, starting with Eq.~\eqref{eq:connect}, is advantageous compared to the preceding one.
\end{rem}

Next, we proceed to general $S_{r+1,n}$ by induction on $r$.
Assume that
\begin{equation}
\avg{\exp(\xi S_{r,n})}\sim C_{r,n}\exp(\varphi^{r-1}(\xi)n),
\label{eq:induction}
\end{equation}
where the coefficient $C_{r,n}$ satisfies
\[
\lim_{n\to\infty}\frac{1}{n}\log C_{r,n}=0,
\]
and we show Eq.~\eqref{eq:induction} for $r+1$.
By Eq.~\eqref{eq:connect} and asymptotic approximation as above, we have
\begin{align*}
\avg{\exp(\xi S_{r+1,n})}&\sim\frac{n!(n-1)!(n-2)!}{(2n-2)!}\sum_{m=1}^{\floor{n/2}}\frac{2^{n-2m}}{(n-2m)!m!(m-1)!}C_{r,m}\exp(\varphi^{r-1}(\xi) m)\\
&\sim\frac{2}{\sqrt{\pi n}}\int_0^{1/2}\frac{C_{r,\beta n}}{\sqrt{1-2\beta}}\exp\left(n g(\beta;\varphi^{r-1}(\xi))\right)d \beta.
\end{align*}
The saddle-point estimation requires to maximize the same function $g$ as in Eq.~\eqref{eq:saddlepoint}, but $\xi$ in Eq.~\eqref{eq:saddlepoint} is replaced by $\varphi^{r-1}(\xi)$ here.
We also note that the coefficient $C_{r,\beta n}$ does not affect the saddle-point method.
Hence, for some coefficient $C_{r+1,n}$, we have
\[
\avg{\exp(\xi S_{r+1,n})}\sim C_{r+1,n}\exp(\varphi^r(\xi)n),
\]
so that
\[
\lim_{n\to\infty}\frac{1}{n}\log\avg{n\exp(\xi S_{r+1,n})}=\varphi^r(\xi).
\]
Thus, the statement holds for any $r$.

\section{Note on approximate forms of the rate function}
Unfortunately, the rate function $I_r(y)$ in Theorem~\ref{thm:main} cannot be expressed exactly for $r\ge2$.
By the definition of the Legendre transformation, $I_r(y)$ is given by
\begin{equation}
I_r(y)=y\xi_r^\ast(y)-\varphi^r(\xi_r^\ast(y)),
\label{eq:defI}
\end{equation}
where $\xi_r^\ast(y)$ satisfies
\[
(\varphi^r)'(\xi_r^\ast(y))=y.
\]
In short, $\xi_r^\ast$ is the inverse function of $(\varphi^r)'$.
The difficulty for $I_r(y)$ is due to the fact that $(\varphi^r)'$ has a complicated form for $r\ge2$ and $\xi_r^\ast(y)$ cannot be solved explicitly.
Instead of the exact form of $I_r(y)$, we derive its approximate forms.

According to the general theory of rate functions~\cite{Ellis}, $I_r(y)$ is convex and takes the minimum value 0 at $y=(\varphi^r)'(0)=4^{-r}$.
Moreover, the derivative of Eq.~\eqref{eq:defI} yields
\begin{equation}
I_r'(y)=\xi_r^\ast(y),\quad
I_r''(y)=(\xi_r^\ast)'(y)=\frac{1}{(\varphi^r)''(\xi_r^\ast(y))}.
\label{eq:derivativeI}
\end{equation}
Owing to $I_r(4^{-r})=0$ and $I_r'(4^{-r})=0$, a second-order Taylor expansion of $I_r$ around $y=4^{-r}$ becomes
\begin{equation}
I_r\left(\frac{1}{4^r}+\eta\right)=\frac{1}{2(\varphi^r)''(0)}\eta^2+O(\eta^3)
=\frac{3\cdot 16^r}{2(4^r-1)}\eta^2+O(\eta^3).
\label{eq:parabola}
\end{equation}
In other words, the bottom of the curve of $I_r(y)$ is approximated by a parabola, and this is equivalent to the central limit theorem~\eqref{eq:cltYamamoto1}.
Differentiating Eq.~\eqref{eq:parabola}, we have
\begin{equation}
\xi_r^\ast\left(\frac{1}{4^r}+\eta\right)=I_r'\left(\frac{1}{4^r}+\eta\right)
=\frac{3\cdot16^r}{4^r-1}\eta+O(\eta^2).
\label{eq:xicenter}
\end{equation}

\begin{figure}\centering
\raisebox{40mm}{(a)}
\includegraphics[scale=0.9]{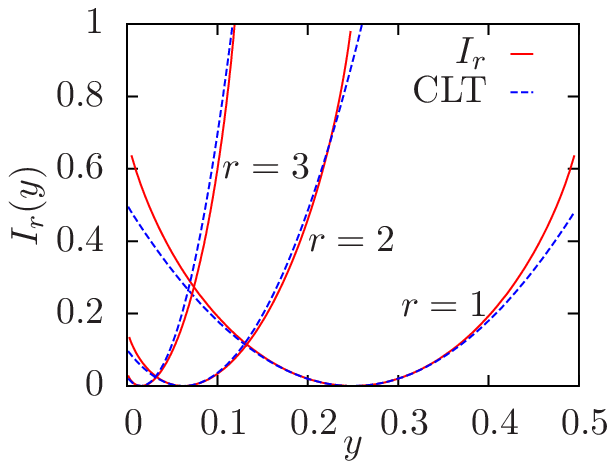}
\hspace{5mm}
\raisebox{40mm}{(b)}
\includegraphics[scale=0.9]{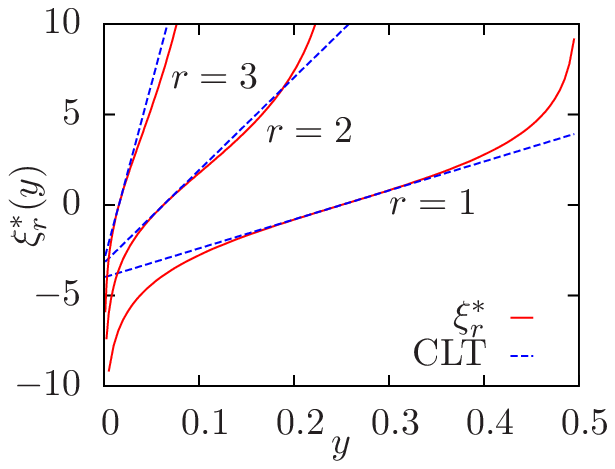}
\caption{
Numerical results. (a) The rate function $I_r(y)$ for $r=1$, 2, and 3 (solid curves)
and parabola~\eqref{eq:parabola} corresponding to the central limit theorem (dashed curves titled ``CLT'').
(b) $\xi_r^\ast(y)$ for $r=1$, 2, and 3 (solid curves) and the approximate line~\eqref{eq:xicenter} around $y=4^{-r}$ at which $I_r(y)$ takes the minimum.
}
\label{fig2}
\end{figure}

Figure~\ref{fig2} shows numerical results of $I_r(y)$ and $\xi_r^\ast(y)$ for $r=1$, 2, and 3 by the solid curves.
We used the Newton-Raphson method to solve $\xi_r^\ast(y)$.
Approximate forms \eqref{eq:parabola} and \eqref{eq:xicenter} corresponding to the central limit theorem are shown by dashed curves and lines.
%Their deviations from the solid curves increase as $y$ becomes away from $4^{-r}$.
The dashed curves are close to the solid ones only in the vicinity of $y=4^{-r}$.
In what follows, we calculate approximate forms of $I_r(y)$ near $y=0$ (leftmost point) and $2^{-r}$ (rightmost point).

\begin{definition}
For simplicity of notation, we introduce
\[
\Psi(X):=\frac{\sqrt{X}+1}{2}
\]
for $X\in[0,\infty)$.
\end{definition}

\begin{prop}
$\varphi$ and $\Psi$ possess the following properties.
\begin{enumerate}
\item $\varphi(\log X)=\log \Psi(X),\quad \varphi^k(\log X)=\log\Psi^k(X)$
\item $\varphi'(\log X)=\frac{\sqrt{X}}{4\Psi(X)}$
\item $\varphi''(\log X)=\frac{\sqrt{X}}{16\Psi(X)^2}$
\end{enumerate}
\label{prop:Psi}
\end{prop}
\begin{proof}
One can easily prove this by using the following relations:
\begin{align*}
\varphi(\xi)&=\frac{\xi}{4}+\log\left(\cosh\frac{\xi}{4}\right)
=\log\left(\frac{e^{\xi/2}+1}{2}\right),\\
\varphi'(\xi)&=\frac{1}{4}+\frac{1}{4}\tanh\frac{\xi}{4}
=\frac{e^{\xi/2}}{2(e^{\xi/2}+1)},\\
\varphi''(\xi)&=\frac{1}{16\cosh^2(\xi/4)}=\frac{e^{\xi/2}}{4(e^{\xi/2}+1)^2}.
\end{align*}
\end{proof}

\begin{prop}
By using $\Psi$, the first and second derivatives of $\varphi^r$ are respectively expressed as
\begin{equation}
(\varphi^r)'(\xi)=\frac{e^{\xi/2}}{4^r}\left[\Psi^r(e^\xi)\prod_{k=1}^r \Psi^k(e^\xi)\right]^{-1/2}
\label{eq:derivativePsi}
\end{equation}
and
\begin{equation}
(\varphi^r)''(\xi)=(\varphi^r)'(\xi)e^{\xi/2}\sum_{l=0}^{r-1}\frac{1}{4^{l+1}\Psi^{l+1}(e^\xi)}\left[\Psi^l(e^\xi)\prod_{k=1}^l \Psi^k(e^\xi)\right]^{-1/2}.
\label{eq:dderivativePsi}
\end{equation}
\label{prop:derivativePsi}
\end{prop}

\begin{proof}
By Eq.~\eqref{eq:dphir} and Proposition~\ref{prop:Psi}, $(\varphi^r)'$ is written as
\begin{align*}
(\varphi^r)'(\xi)
&=\varphi'(\log\Psi^{r-1}(e^\xi))\varphi'(\log\Psi^{r-2}(e^\xi))\cdots\varphi'(\log e^\xi)\\
&=\frac{\sqrt{\Psi^{r-1}(e^\xi)}}{4\Psi^r(e^\xi)}\frac{\sqrt{\Psi^{r-2}(e^\xi)}}{4\Psi^{r-1}(e^\xi)}\cdots\frac{e^{\xi/2}}{4\Psi(e^\xi)}\\
&=\frac{e^{\xi/2}}{4^r}\frac{1}{\Psi^r(e^\xi)\sqrt{\Psi(e^\xi)\cdots\Psi^{r-1}(e^\xi)}}.
\end{align*}
Next, we get Eq.~\eqref{eq:dderivativePsi} straightforwardly by Eq.~\eqref{eq:ddphir} and Proposition~\ref{prop:Psi}.
\end{proof}

%\begin{rem}
%Obviously, $(\varphi^0)(\xi)=\xi$, so that $(\varphi^0)'(\xi)=1$ and $(\varphi^0)''(\xi)=0$.
%Proposition~\ref{prop:derivativePsi} is valid also for $r=0$, provided that
%\[
%\prod_{k=1}^0\cdot=1,\quad\sum_{l=0}^{-1}\cdot=0.
%\]
%\end{rem}

Using the above properties of $\varphi^r$, let us derive the expansion of $\xi_r^\ast(y)$.
\begin{theorem}[Asymptotic forms of $\xi_r^\ast(y)$ around $y=0$ and $y=2^{-r}$]
\begin{enumerate}
\item Around $y=0$ which is the leftmost point of $I_r(y)$,
\[
\xi_r^\ast(\eta)=2\log\left(4^r\eta\sqrt{\Psi^r(0)\prod_{k=1}^{r}\Psi^k(0)}\right)+O(\eta).
\]
\item Around $y=2^{-r}$ which is the rightmost point of $I_r(y)$,
\[
\xi_r^\ast\left(\frac{1}{2^r}-\eta\right)=-2^r\log\left(2^{r-1+2^{1-r}}\eta\right)+O(\eta).
\]
\end{enumerate}
\label{thm:expansionxi}
\end{theorem}

\begin{proof}
By Eq.~\eqref{eq:derivativeI},
\begin{equation}
(\varphi^r)''(\xi_r^\ast(y))(\xi_r^\ast)'(y)=1.
\label{eq:DExi}
\end{equation}
In this proof, we use this formula as a differential equation to determine $\xi_r^\ast$.
Since $\xi_r^\ast(0)=-\infty$ and $\xi_r^\ast(2^{-r})=\infty$,
we need to expand $(\varphi^r)''(\xi)$ around $\xi=\mp\infty$, corresponding to $y=0$ and $2^{-r}$, by means of Proposition~\ref{prop:derivativePsi}.
\begin{enumerate}
\item %1
We expand $(\varphi^r)''(\xi)$ on condition that $e^\xi$ is sufficiently small.
By definition, $\Psi(e^\xi)=(e^{\xi/2}+1)/2=\Psi(0)+O(e^{\xi/2})$, and similarly $\Psi^k(e^\xi)=\Psi^k(0)+O(e^{\xi/2})$ for $k\ge1$.
Putting into Eq.~\eqref{eq:derivativePsi}, we immediately obtain
\[
(\varphi^r)'(\xi)=\frac{e^{\xi/2}}{4^r}\left[\Psi^r(0)\prod_{k=1}^r\Psi^k(0)\right]^{-1/2}+O(e^\xi)
=: \frac{e^{\xi/2}}{Q_r}+O(e^\xi).
\]
To estimate the sum in Eq.~\eqref{eq:dderivativePsi}, the term of $l=0$ is $O(e^{-\xi/2})$ which is the leading order, and the others are $O(1)$.
Thus, by neglecting the terms other than $l=0$, we have
\[
(\varphi^r)''(\xi)=(\varphi^r)'(\xi)e^{\xi/2}\left(\frac{1}{2}e^{-\xi/2}+O(1)\right).
\]
Hence, Eq.~\eqref{eq:DExi} becomes
\[
\left(\frac{1}{2Q_r}\exp(\xi_r^\ast/2)+O(\exp(\xi_r^\ast))\right)\frac{d\xi_r^\ast}{d y}=1.
\]
By integrating from $y=0$ to $\eta$,
\[
\exp(\xi_r^\ast(\eta)/2)+O(\exp(\xi_r^\ast(\eta)))=Q_r\eta,
\]
and the solution $\xi_r^\ast$ is
\[
\xi_r^\ast(\eta)=2\log(Q_r\eta)+O(\eta).
\]
%%%%%%%%
\item %2
Contrary to the above, $\xi_r^\ast(y)$ tends to infinity as $y\nearrow2^{-r}$, so we need to expand $(\varphi^r)''(\xi)$ when $e^\xi$ is sufficiently large.
From the observation
\[
\Psi(e^{\xi})=\frac{e^{\xi/2}}{2}+O(1),\quad
\Psi^2(e^{\xi})=\frac{\sqrt{\Psi(e^\xi)}+1}{2}=\frac{e^{\xi/4}}{2^{3/2}}+O(1),
\]
we reasonably set $\Psi^k(e^{\xi})=2^{-\rho_k}\exp(\xi/2^k)+O(1)$.
The exponent $\rho_k$ satisfies $\rho_1=1$ and $\rho_{k+1}=\rho_k/2+1$, so that
\[
\rho_k=2-2^{1-k}.
\]
Noting that
\[
\rho_r+\sum_{k=1}^r \rho_k=2r,
\]
we have
\[
\left[\Psi^r(e^\xi)\prod_{k=1}^r \Psi^k(e^\xi)\right]^{-1/2}=2^r e^{-\xi/2}(1+O(e^{-\xi/2^r})).
\]
In this case, the dominant term of the sum in Eq.~\eqref{eq:dderivativePsi} corresponds to $l=r-1$, so that
\[
(\varphi^r)''(\xi)=2^{\rho_r-1-2r}e^{-\xi/2^r}(1+O(e^{-\xi/2^r})).
\]
Finally, integrating Eq.~\eqref{eq:DExi} from $y=2^{-r}-\eta$ to $2^{-r}$ as above, we get
\[
\xi_r^\ast\left(\frac{1}{2^r}-\eta\right)=-2^r\log\left(2^{r-1+2^{1-r}}\eta\right)+O(\eta).
\]
\end{enumerate}
\end{proof}

By using $\xi_r^\ast(y)$ in Theorem~\ref{thm:expansionxi}, we reach asymptotic forms of $I_r(y)$.
\begin{theorem}[Asymptotic forms of the rate function $I_r(y)$ around $y=0$ and $2^{-r}$]
\begin{enumerate}
\item Around $y=0$,
\[
I_r(\eta)=2\eta\log\left(4^r\eta\sqrt{\Psi^r(0)\prod_{k=1}^r\Psi^k(0)}\right)-\log\Psi^r(0)-2\eta+O(\eta^2).
\]
\item Around $y=2^{-r}$,
\[
I_r\left(\frac{1}{2^r}-\eta\right)=2^r\eta\log\left(2^{r-1+2^{1-r}}\eta\right)+(2-2^{1-r})\log2-2^r\eta+O(\eta^2).
\]
\end{enumerate}
\label{thm:expansionI}
\end{theorem}

\begin{proof}
\begin{enumerate}
\item As in the proof of Proposition~\ref{prop:Psi}, we use $Q_r=4^r[\Psi^r(0)\prod_{k=1}^r\Psi^k(0)]^{1/2}$ and
\[
I_r'(\eta)=\xi_r^\ast(\eta)=2\log(Q_r\eta)+O(\eta).
\]
By integrating from $0$ to $\eta$,
\[
I_r(\eta)-I_r(0)=2\eta\log(Q_r\eta)-2\eta+O(\eta^2).
\]
The proof is completed by calculating $I_r(0)$ as
\begin{align*}
I_r(0)
&=\lim_{y\to0}\left(2y\log(Q_ry)+O(y^2)-\varphi^r(\xi_r^\ast(y))\right)\\
&=-\lim_{\xi\to-\infty}\varphi^r(\xi)\\
&=-\lim_{\xi\to-\infty}\log\Psi^r(e^\xi)\\
&=-\log\Psi^r(0).
\end{align*}
%%%%
\item
By integrating Eq.~\eqref{eq:DExi} from $2^{-r}-\eta$ to $2^{-r}$, we have
\[
I_r\left(\frac{1}{2^r}-\eta\right)-I_r\left(\frac{1}{2^r}\right)
=2^r\eta\log\left(2^{r-1+2^{1-r}}\eta\right)-2^r\eta+O(\eta^2).
\]
We need to be careful in the calculation of $I_r(2^{-r})$, because both $\xi_r^\ast(y)$ and $\varphi^r(\xi_r^\ast(y))$ in Eq.~\eqref{eq:defI} diverge as $y\to2^{-r}$.
\begin{align*}
I_r\left(\frac{1}{2^r}\right)
&=\lim_{\eta\to0}\left[\left(\frac{1}{2^r}-\eta\right)\xi_r^\ast\left(\frac{1}{2^r}-\eta\right)-\varphi^r\left(\xi_r^\ast\left(\frac{1}{2^r}-\eta\right)\right)\right]\\
%&=\lim_{\eta\to0}\left[\frac{1}{2^r}\xi_r^\ast\left(\frac{1}{2^r}-\eta\right)-\varphi^r\left(\xi_r^\ast\left(\frac{1}{2^r}-\eta\right)\right)\right]\\
&=\lim_{\xi\to\infty}\left(\frac{\xi}{2^r}-\varphi^r(\xi)\right)\\
&=\lim_{\xi\to\infty}\left(\frac{\xi}{2^r}-\log\Psi^r(e^\xi)\right)\\
&=\lim_{\xi\to\infty}\left(\frac{\xi}{2^r}-\log\left(2^{-\rho_r}e^{\xi/2^r}+O(1)\right)\right)\\
&=\rho_r\log2=(2-2^{1-r})\log2.
\end{align*}
\end{enumerate}
\end{proof}

\begin{rem}
The same calculation applies to $y=4^{-r}$ ($\xi_r^\ast(4^{-r})=0$).
Note that $e^0=1$ is the fixed point of $\Psi$,  namely $\Psi(1)=1$, so we obtain
\[
%(\varphi^r)'(\xi)=\frac{1}{4^r}+O(\xi),\quad
(\varphi^r)''(\xi)=\frac{4^r-1}{3\cdot16^r}+O(\xi).
\]
The Taylor expansion~\eqref{eq:parabola} of $I_r(y)$ around $y=4^{-r}$ is reproduced.
\end{rem}

By the exact form of $I_1(y)=I(y)$ in Eq.~\eqref{eq:rate}, $I_1(y)$ is symmetric about $y=1/4$.
On the other hand, comparing the approximate forms of $I_r(y)$ near $y=0$ and $2^{-r}$ in Theorem~\ref{thm:expansionI}, $I_r(y)$ for $r\ge2$ is clearly asymmetric.

In Fig.~\ref{fig3}, we show numerical results of $\xi_r^\ast(y)$ and $I_r(y)$ along with the asymptotic forms at $y\simeq0$ (dashed curves) and $y\simeq2^{-r}$ (dot-dashed curves) from Theorems~\ref{thm:expansionxi} and \ref{thm:expansionI}.
(The solid curves are the same as in Fig.~\ref{fig2}.)

\begin{figure}[t!]\centering
\raisebox{40mm}{(a)}
\includegraphics[scale=0.9]{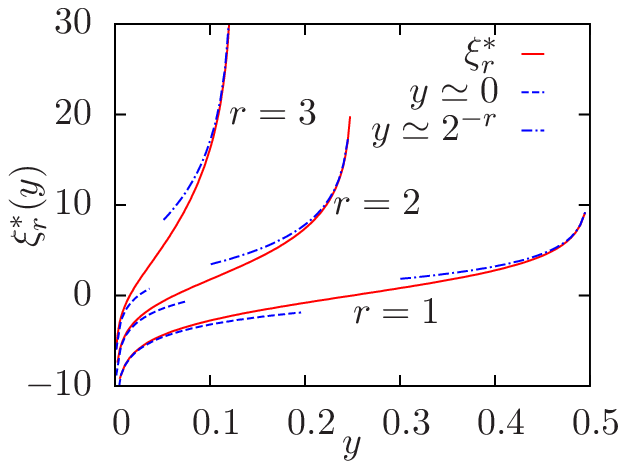}
\raisebox{40mm}{(b)}
\includegraphics[scale=0.9]{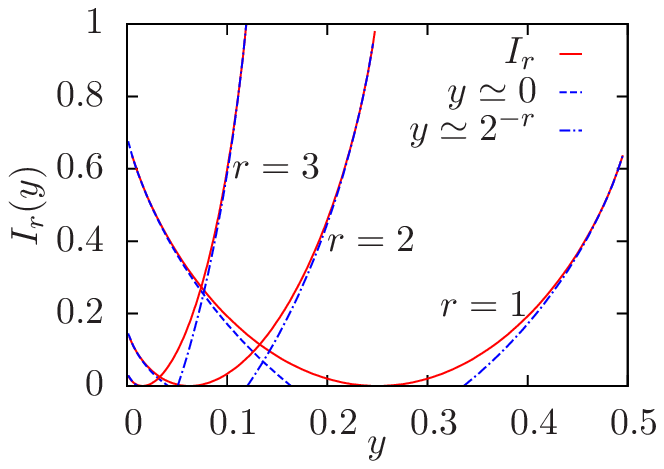}
\caption{
Numerical results.
The solid curves in (a) and (b) are respectively the rate function $I_r(y)$ and $\xi_r^\ast(y)$ for $r=1,2,3$, and the dashed and dot-dashed curves respectively represent the approximate curves at $y\simeq0$ and $y\simeq2^{-r}$.
}
\label{fig3}
\end{figure}

\section*{Acknowledgments}
The author is grateful to referees for instructing recent related articles.
The idea that Eq.~\eqref{eq:cltYamamoto2} is derived by the pruning operation is suggested by a referee.
The present work was partially supported by a University of the Ryukyus Research Project Promotion Grant for Young Researchers (17SP04109), and Hayao Nakayama Foundation for Science \& Technology and Culture (H29-B-41).

%\bibliographystyle{siamplain}
%\bibliography{references}

\end{document}